\documentclass[12pt,reqno,a4paper]{amsart}%
\usepackage{amsfonts}
\usepackage{amsmath}
\usepackage{amssymb}
\usepackage{graphicx}%
\providecommand{\U}[1]{\protect \rule{.1in}{.1in}}
\newtheorem{theorem}{Theorem}[section]
\newtheorem{corollary}[theorem]{Corollary}
\newtheorem{lemma}[theorem]{Lemma}
\theoremstyle{definition}

\theoremstyle{example}

\theoremstyle{remark}
\newtheorem{remark}[theorem]{Remark}
\begin{document}
\title[Partial sums of generalized Rabotnov function]{Partial sums of generalized Rabotnov function}
\author[B. A. Frasin]{Basem Aref Frasin}
\address{Department of Mathematics, Faculty of Science, Al al-Bayt University, Mafraq,
Jordan. \\
ORCID:https://orcid.org/0000-0001-8608-8063 }
\email{bafrasin@yahoo.com}

\begin{abstract}
Let $(\mathbb{R}_{\alpha,\beta,\gamma}(z))_{m}(z)=z+\sum_{n=1}^{m}A_{n}%
z^{n+1}$ be the sequence of partial sums of the normalized Rabotnov function
$\mathbb{R}_{\alpha,\beta,\gamma}(z)=z+\sum_{n=1}^{\infty}A_{n}z^{n+1}$ where
$A_{n}=\frac{\beta^{n}\Gamma \left(  \gamma+\alpha \right)  }{\Gamma \left(
\left(  \gamma+\alpha \right)  (n+1)\right)  }.$ The purpose of the present
paper is to determine lower bounds for $\mathfrak{R}\left \{  \frac
{\mathbb{R}_{\alpha,\beta,\gamma}(z)}{(\mathbb{R}_{\alpha,\beta,\gamma}%
)_{m}(z)}\right \}  ,\mathfrak{R}\left \{  \frac{(\mathbb{R}_{\alpha
,\beta,\gamma})_{m}(z)}{\mathbb{R}_{\alpha,\beta,\gamma}(z)}\right \}  ,$

$\mathfrak{R}\left \{  \frac{\mathbb{R}_{\alpha,\beta,\gamma}(z)}%
{(\mathbb{R}_{\alpha,\beta,\gamma})_{m}^{\prime}(z)}\right \}  ,\mathfrak{R}%
\left \{  \frac{(\mathbb{R}_{\alpha,\beta,\gamma})_{m}^{\prime}(z)}%
{\mathbb{R}_{\alpha,\beta,\gamma}(z)}\right \}  .$ Furthermore, we give lower
bounds for $\mathfrak{R}\left \{  \frac{\mathbb{I}\left[  \mathbb{R}%
_{\alpha,\beta,\gamma}\right]  (z)}{(\mathbb{I}\left[  \mathbb{R}%
_{\alpha,\beta,\gamma}\right]  )_{m}(z)}\right \}  $ and $\mathfrak{R}\left \{
\frac{(\mathbb{I}\left[  \mathbb{R}_{\alpha,\beta,\gamma}\right]  )_{m}%
(z)}{\mathbb{I}\left[  \mathbb{R}_{\alpha,\beta,\gamma}\right]  (z)}\right \}
$ where $\mathbb{I}\left[  \mathbb{R}_{\alpha,\beta,\gamma}\right]  $ is the
Alexander transform of $\mathbb{R}_{\alpha,\beta,\gamma}$. Several examples of
the main results are also considered.



\end{abstract}
\maketitle

\section{Introduction and preliminaries}

In 1948, Rabotnov \cite{rab} introduced a special function applied in
viscoelasticity. This function, known today as the Rabotnov fractional
exponential function or briefly Rabotnov function, is defined as
follows$\allowbreak$%
\begin{equation}
R_{\alpha,\beta}(z)=z^{\alpha}\sum \limits_{n=0}^{\infty}\frac{(\beta
)^{n}z^{n\left(  1+\alpha \right)  }}{\Gamma(\left(  n+1\right)  \left(
1+\alpha \right)  )},\quad(\alpha,\beta,z\in \mathbb{C}\;). \label{rand}%
\end{equation}
Rabotnov function is the particular case of the familiar Mittag-Leffler
function \cite{mit} widely used in the solution of fractional order integral
equations or fractional order differential equations. The relation between the
Rabotnov function and Mittag-Leffler function can be written as follows%

\[
R_{\alpha,\beta}(z)=z^{\alpha}E_{1+\alpha,1+\alpha}(\beta z^{1+\alpha}),
\]
where $E$ is Mittag-Leffler function and $\alpha,\beta,z\in \mathbb{C}$.
Several properties of Mittag-Leffler function and generalized Mittag-Leffler
function can be found in \cite{frtr,at,ban,mur,ga}.

A generalization of the Rabotnov function $R_{\alpha,\beta}$ is given by%

\begin{equation}
R_{\alpha,\beta,\gamma}(z)=z^{\alpha}\sum \limits_{n=0}^{\infty}\frac
{(\beta)^{n}z^{n\left(  \gamma+\alpha \right)  }}{\Gamma(\left(  n+1\right)
\left(  \gamma+\alpha \right)  )},\quad(\alpha,\beta,\gamma,z\in \mathbb{C}\;).
\label{gg}%
\end{equation}
Let $\mathcal{A}$ denote the class of functions of the form%

\begin{equation}
f(z)=z+\sum \limits_{n=2}^{\infty}a_{n}z^{n} \label{b1}%
\end{equation}
which are analytic in the open unit disc $\mathfrak{U}=\left \{  z:\left \vert
z\right \vert <1\right \}  $ and hold the normalization condition
$f(0)=f^{\prime}(0)-1=0.$ Further, by $\mathcal{S}$ we shall denote the class
of all functions in $\mathcal{A}$ which are univalent in $\mathfrak{U}$.

The Alexander transform $\mathbb{I}[f]:\mathfrak{U}\rightarrow \mathbb{C}$ of
$f$ \ is defined by \cite{alx}%
\[
\mathbb{I}[f]=%
{\displaystyle \int \limits_{0}^{z}}
\frac{f(t)}{t}dt=z+\sum \limits_{n=2}^{\infty}\frac{a_{n}}{n}z^{n}.
\]
It is clear that the Rabotnov function $R_{\alpha,\beta,\gamma}(z)$ does not
belong to the family $\mathcal{A}$. Thus, it is natural to consider the
following normalization of Rabotnov function%
\begin{align}
\mathbb{R}_{\alpha,\beta,\gamma}(z)  &  =z^{\frac{\gamma}{\gamma+\alpha}%
}\Gamma \left(  \gamma+\alpha \right)  R_{\alpha,\beta,\gamma}(z^{\frac
{1}{\gamma+\alpha}})\nonumber \\
&  =z+\sum \limits_{n=1}^{\infty}\frac{\beta^{n}\Gamma \left(  \gamma
+\alpha \right)  }{\Gamma \left(  \left(  \gamma+\alpha \right)  (n+1)\right)
}z^{n+1},\;z\in \mathfrak{U}. \label{GT}%
\end{align}
Geometric properties including starlikeness, convexity and close-to-convexity
for the normalized Rabotnov function $\mathbb{R}_{\alpha,\beta,1}(z)$ were
recently investigated by Eker and Ece in \cite{ek}. A new class of normalized
analytic functions and bi-univalent functions associated with the normalized
Rabotnov function $\mathbb{R}_{\alpha,\beta,1}(z)$ was also introduced and
studied by Amourah et al.\cite{am}.

\begin{remark}
\label{rem1}The function $\mathbb{R}_{\alpha,\beta,\gamma}(z)$ contains many
well known functions as its special case, for example:
\end{remark}

$\left \{
\begin{array}
[c]{c}%
\mathbb{R}_{0,-\frac{1}{3},1}(z)=ze^{-\frac{z}{3}}%
,\  \  \  \  \  \  \  \  \  \  \  \  \  \\
\mathbb{R}_{1,\frac{1}{2},1}(z)=\sqrt{2z}\sinh \sqrt{\frac{z}{2}},\\
\mathbb{R}_{1,-\frac{1}{4},1}(z)=2\sqrt{z}\sin \frac{\sqrt{z}}{2},\\
\mathbb{R}_{1,1,1}(z)=\sqrt{z}\sinh \sqrt{z},\  \  \  \  \  \  \\
\mathbb{R}_{1,2,1}(z)=\frac{1}{2}\sqrt{2z}\sinh \sqrt{2z}.
\end{array}
\right.  $\noindent

\bigskip The concept of finding the lower bound of the real part of the ratio
of the partial sum of analytic functions to its infinite series sum was
introduced firstly by Silvia \cite{siv}. Silverman in \cite{sil} found the
partial sums of convex and starlike functions by developed more useful
techniques. After that, several researchers investigated such partial sums for
different subclasses of analytic functions. For more work on partial sums, the
interested readers are referred to \cite{bri,fr1,fr2,mur,lin,or3,owa,sh}.

Recently, some researchers have studied on partial sums of special functions.
For example, Orhan and Yagmur in \cite{yag} determined lower bounds for the
normalized Struve functions to its sequence of partial sums. Some lower bounds
for the quotients of normalized Dini functions and their partial sum, as well
as for the quotients of the derivative of normalized Dini functions and their
partial sums were obtained by Akta\c{s} and Orhan in \cite{ak}. Din et al.
\cite{din} found the partial sums of two kinds normalized Wright functions and
the partial sums of Alexander transform of these normalized Wright functions.
\bigskip Kaz\i mo\u{g}lu in \cite{kaz} studied the partial sums of the
normalized Miller-Ross function. Kaz\i mo\u{g}lu and Deniz \cite{kaz2}
determined lower bounds for the normalized Rabotnov function $\mathbb{R}%
_{\alpha,\beta,1}(z)$ to its sequence of partial sums.

Let $(\mathbb{R}_{\alpha,\beta,\gamma}(z))_{m}(z)=z+\sum_{n=1}^{m}\frac
{\beta^{n}\Gamma \left(  \gamma+\alpha \right)  }{\Gamma \left(  \left(
\gamma+\alpha \right)  (n+1)\right)  }z^{n+1},m\in \mathbb{N=\{}1,2,3,...\},$ be
the sequence of partial sums of normalized Rabotnov function $\mathbb{R}%
_{\alpha,\beta,\gamma}(z)$ given by (\ref{GT}) and for $m=0$, we have
$(\mathbb{R}_{\alpha,\beta,\gamma}(z))_{0}(z)=z$.

The aim of the present paper is to determine the lower bounds of%

\begin{align*}
&  \mathfrak{R}\left \{  \frac{\mathbb{R}_{\alpha,\beta,\gamma}(z)}%
{(\mathbb{R}_{\alpha,\beta,\gamma})_{m}(z)}\right \}  ,\mathfrak{R}\left \{
\frac{(\mathbb{R}_{\alpha,\beta,\gamma})_{m}(z)}{\mathbb{R}_{\alpha
,\beta,\gamma}(z)}\right \}  ,\mathfrak{R}\left \{  \frac{\mathbb{R}%
_{\alpha,\beta,\gamma}^{\prime}(z)}{(\mathbb{R}_{\alpha,\beta,\gamma}%
)_{m}^{\prime}(z)}\right \}  ,\mathfrak{R}\left \{  \frac{(\mathbb{R}%
_{\alpha,\beta,\gamma})_{m}^{\prime}(z)}{\mathbb{R}_{\alpha,\beta,\gamma
}^{\prime}(z)}\right \}  ,\\
&  \mathfrak{R}\left \{  \frac{\mathbb{I}\left[  \mathbb{R}_{\alpha
,\beta,\gamma}\right]  (z)}{(\mathbb{I}\left[  \mathbb{R}_{\alpha,\beta
,\gamma}\right]  )_{m}(z)}\right \}  ,\mathfrak{R}\left \{  \frac{(\mathbb{I}%
\left[  \mathbb{R}_{\alpha,\beta,\gamma}\right]  )_{m}(z)}{\mathbb{I}\left[
\mathbb{R}_{\alpha,\beta,\gamma}\right]  (z)}\right \}  ,
\end{align*}
where $\mathbb{I}\left[  \mathbb{R}_{\alpha,\beta,\gamma}\right]  $ is the
Alexander transform of $\mathbb{R}_{\alpha,\beta,\gamma}$.

Throughout this paper, we shall restrict our attention to the case of
real-valued $\alpha \geq0,\gamma \geq1,$ $\beta \in \mathbb{C}$ and $z\in
\mathfrak{U}$.

In order to obtain our results we need the following lemmas.

\begin{lemma}
\label{lem1}If $n\in \mathbb{N}$ and $\alpha \geq0,\gamma \geq1,$ then
\begin{equation}
(\gamma+\alpha)^{n-1}(n-1)!\Gamma \left(  \gamma+\alpha \right)  \leq
\Gamma \left(  (\gamma+\alpha \right)  n). \label{gt}%
\end{equation}

\end{lemma}

\begin{proof}
Using the inductive method, we can easily prove the inequality (\ref{gt}).
\end{proof}

\begin{lemma}
\label{lem2}Let $\alpha \geq0,\gamma \geq1$ and $\beta \in \mathbb{C}$. Then the
function $\mathcal{\ }\mathbb{R}_{\alpha,\beta,\gamma}:\mathfrak{U\rightarrow
}\mathbb{C}$ defined by (\ref{GT}) satisfies the following inequalities:
\end{lemma}

$(i)$ If $2\left(  \gamma+\alpha \right)  >\left \vert \beta \right \vert $, then
\[
\left \vert \mathbb{R}_{\alpha,\beta,\gamma}(z)\right \vert \leq \frac{2\left(
\gamma+\alpha \right)  +\left \vert \beta \right \vert }{2\left(  \gamma
+\alpha \right)  -\left \vert \beta \right \vert }\  \  \  \  \  \left(  z\in
\mathfrak{U}\right)  .
\]

$(ii)$If $\gamma+\alpha>\left \vert \beta \right \vert ,$ then
\[
\left \vert \mathbb{R}_{\alpha,\beta,\gamma}^{\prime}(z)\right \vert \leq
\frac{\gamma+\alpha+\left \vert \beta \right \vert }{\gamma+\alpha-\left \vert
\beta \right \vert }\  \  \  \  \  \left(  z\in \mathfrak{U}\right)  .
\]

$(iii)$ If $2\left(  \gamma+\alpha \right)  >\left \vert \beta \right \vert ,$
then
\[
\left \vert \mathbb{I}[\mathbb{R}_{\alpha,\beta,\gamma}(z)]]\right \vert
\leq \frac{2(\gamma+\alpha)}{2(\gamma+\alpha)-\left \vert \beta \right \vert
}\  \  \  \  \  \left(  z\in \mathfrak{U}\right)  .
\]

\begin{proof}
(i) By using the inequality (\ref{gt}) of Lemma \ref{lem2} and the inequality
\[
n!\geq2^{n-1}\  \left(  n\in \mathbb{N}\right)  ,
\]
we have%
\begin{align*}
\left \vert \mathbb{R}_{\alpha,\beta,\gamma}(z)\right \vert  &  =\left \vert
z+\sum \limits_{n=1}^{\infty}\frac{\beta^{n}\Gamma \left(  \gamma+\alpha \right)
}{\Gamma \left(  \left(  \gamma+\alpha \right)  (n+1)\right)  }z^{n+1}%
\right \vert \\
&  \leq1+\sum \limits_{n=1}^{\infty}\frac{\left \vert \beta \right \vert
^{n}\Gamma \left(  \gamma+\alpha \right)  }{\Gamma \left(  \left(  \gamma
+\alpha \right)  (n+1)\right)  }\\
&  \leq1+\sum_{n=1}^{\infty}\frac{\left \vert \beta \right \vert ^{n}}{\left(
\gamma+\alpha \right)  ^{n}n!}\\
&  \leq1+\frac{\left \vert \beta \right \vert }{\left(  \gamma+\alpha \right)
}\sum_{n=1}^{\infty}\left(  \frac{\left \vert \beta \right \vert }{2\left(
\gamma+\alpha \right)  }\right)  ^{n-1}\\
&  =\frac{2\left(  \gamma+\alpha \right)  +\left \vert \beta \right \vert
}{2\left(  \gamma+\alpha \right)  -\left \vert \beta \right \vert },\  \left(
2\left(  \gamma+\alpha \right)  >\left \vert \beta \right \vert \right)  .\
\end{align*}
$\allowbreak$ $\allowbreak$

(ii) To prove (ii), using the inequality (\ref{gt}) of Lemma \ref{lem1}, and
the inequality%
\[
2n!\geq n+1\  \  \  \left(  n\in \mathbb{N}\right)  ,
\]
we have
\begin{align*}
\left \vert \mathbb{R}_{\alpha,\beta,\gamma}^{\prime}(z)\right \vert  &
=\left \vert 1+\sum \limits_{n=1}^{\infty}\frac{(n+1)\beta^{n}\Gamma \left(
\gamma+\alpha \right)  }{\Gamma \left(  \left(  \gamma+\alpha \right)
(n+1)\right)  }z^{n}\right \vert \\
&  \leq1+\sum \limits_{n=1}^{\infty}\frac{(n+1)\left \vert \beta \right \vert
^{n}\Gamma \left(  \gamma+\alpha \right)  }{\Gamma \left(  \left(  \gamma
+\alpha \right)  (n+1)\right)  }\\
&  \leq1+\sum_{n=1}^{\infty}\frac{(n+1)\left \vert \beta \right \vert ^{n}%
}{\left(  \gamma+\alpha \right)  ^{n}n!}\\
&  \leq1+\frac{2\left \vert \beta \right \vert }{\gamma+\alpha}\sum_{n=1}%
^{\infty}\left(  \frac{\left \vert \beta \right \vert }{\gamma+\alpha}\right)
^{n-1}\\
&  =\frac{\gamma+\alpha+\left \vert \beta \right \vert }{\gamma+\alpha-\left \vert
\beta \right \vert },\  \left(  \gamma+\alpha>\left \vert \beta \right \vert
\right)  .
\end{align*}

(iii) Making the use of the inequality (\ref{gt}) of Lemma \ref{lem2} and the
inequality
\[
(n+1)!\geq2^{n}\  \  \  \  \left(  n\in \mathbb{N}\right)  ,
\]

we thus find%
\begin{align*}
\left \vert \mathbb{I}[\mathbb{R}_{\alpha,\beta,\gamma}](z)]\right \vert  &
=\left \vert z+\sum \limits_{n=1}^{\infty}\frac{\beta^{n}\Gamma \left(
\gamma+\alpha \right)  }{(n+1)\Gamma \left(  \left(  \gamma+\alpha \right)
(n+1)\right)  }z^{n+1}\right \vert \\
&  \leq1+\sum_{n=1}^{\infty}\frac{\left \vert \beta \right \vert ^{n}}{\left(
\gamma+\alpha \right)  ^{n}(n+1)!}\\
&  \leq1+\frac{\left \vert \beta \right \vert }{2(\gamma+\alpha)}\sum
_{n=1}^{\infty}\left(  \frac{\left \vert \beta \right \vert }{2(\gamma+\alpha
)}\right)  ^{n-1}\\
&  =\frac{2(\gamma+\alpha)}{2(\gamma+\alpha)-\left \vert \beta \right \vert
},\  \left(  2(\gamma+\alpha)>\left \vert \beta \right \vert \right)  .
\end{align*}

\end{proof}

Let $w(z)$ be an analytic function in $\mathfrak{U}$: In the sequel, we will
use the following well-known result:%
\[
\mathfrak{R}\left \{  \frac{1+w(z)}{1-w(z)}\right \}  >0,\text{\ }%
z\in \mathfrak{U}\text{ if and only if }|w(z)|<1,\text{\ }z\in \mathfrak{U}%
\text{.}%
\]

\begin{theorem}
\label{th1}Let $\alpha \geq0,\gamma \geq1$ and $2\left(  \gamma+\alpha \right)
\geq3\left \vert \beta \right \vert $. Then
\end{theorem}

\begin{equation}
\mathfrak{R}\left \{  \frac{\mathbb{R}_{\alpha,\beta,\gamma}(z)}{(\mathbb{R}%
_{\alpha,\beta,\gamma})_{m}(z)}\right \}  \geq \frac{2\left(  \gamma
+\alpha \right)  -3\left \vert \beta \right \vert }{2\left(  \gamma+\alpha \right)
-\left \vert \beta \right \vert },\text{ \  \  \  \  \  \ }z\in \mathfrak{U}, \label{t}%
\end{equation}

and%

\begin{equation}
\mathfrak{R}\left \{  \frac{(\mathbb{R}_{\alpha,\beta,\gamma})_{m}%
(z))}{\mathbb{R}_{\alpha,\beta,\gamma}(z)}\right \}  \geq \frac{2\left(
\gamma+\alpha \right)  -\left \vert \beta \right \vert }{2\left(  \gamma
+\alpha \right)  +\left \vert \beta \right \vert },\text{ \  \  \  \  \  \  \ }%
z\in \mathfrak{U}. \label{tt}%
\end{equation}

\begin{proof}
From inequality (i) of Lemma \ref{lem2}, we get%
\[
1+\sum_{n=1}^{\infty}\left \vert A_{n}\right \vert \leq \frac{2\left(
\gamma+\alpha \right)  +\left \vert \beta \right \vert }{2\left(  \gamma
+\alpha \right)  -\left \vert \beta \right \vert },
\]
or equivalently%
\[
\left(  \frac{2\left(  \gamma+\alpha \right)  -\left \vert \beta \right \vert
}{2\left \vert \beta \right \vert }\right)  \sum_{n=1}^{\infty}\left \vert
A_{n}\right \vert \leq1,
\]

where $A_{n}=\frac{\beta^{n}\Gamma \left(  \gamma+\alpha \right)  }%
{\Gamma \left(  \left(  \gamma+\alpha \right)  (n+1)\right)  }.$

In order to prove the inequality (\ref{t}), we consider the function $w(z)$
defined by%

\begin{align}
\frac{1+w(z)}{1-w(z)}  &  =\left(  \frac{2\left(  \gamma+\alpha \right)
-\left \vert \beta \right \vert }{2\left \vert \beta \right \vert }\right)  \left[
\frac{\mathbb{R}_{\alpha,\beta,\gamma}(z)}{(\mathbb{R}_{\alpha,\beta,\gamma
})_{m}(z)}-\frac{2\left(  \gamma+\alpha \right)  -3\left \vert \beta \right \vert
}{2\left(  \gamma+\alpha \right)  -\left \vert \beta \right \vert }\right]
\nonumber \\
&  =\frac{1+\sum_{n=1}^{m}A_{n}z^{n}+\left(  \frac{2\left(  \gamma
+\alpha \right)  -\left \vert \beta \right \vert }{2\left \vert \beta \right \vert
}\right)  \sum_{n=m+1}^{\infty}A_{n}z^{n}}{1+\sum_{n=1}^{m}A_{n}z^{n}}.
\label{ss}%
\end{align}
Now, from (\ref{ss}) we can write%
\[
w(z)=\frac{\left(  \frac{2\left(  \gamma+\alpha \right)  -\left \vert
\beta \right \vert }{2\left \vert \beta \right \vert }\right)  \sum_{n=m+1}%
^{\infty}A_{n}z^{n}}{2+2\sum_{n=1}^{m}A_{n}z^{n}+\left(  \frac{2\left(
\gamma+\alpha \right)  -\left \vert \beta \right \vert }{2\left \vert
\beta \right \vert }\right)  \sum_{n=m+1}^{\infty}A_{n}z^{n}}%
\]
and%
\[
\left \vert w(z)\right \vert \leq \frac{\left(  \frac{2\left(  \gamma
+\alpha \right)  -\left \vert \beta \right \vert }{2\left \vert \beta \right \vert
}\right)  \sum_{n=m+1}^{\infty}\left \vert A_{n}\right \vert }{2-2\sum_{n=1}%
^{m}\left \vert A_{n}\right \vert -\left(  \frac{2\left(  \gamma+\alpha \right)
-\left \vert \beta \right \vert }{2\left \vert \beta \right \vert }\right)
\sum_{n=m+1}^{\infty}\left \vert A_{n}\right \vert }.
\]
This implies that $|w(z)|\leq1$ if and only if%
\[
2\left(  \frac{2\left(  \gamma+\alpha \right)  -\left \vert \beta \right \vert
}{2\left \vert \beta \right \vert }\right)  \sum_{n=m+1}^{\infty}\left \vert
A_{n}\right \vert \leq2-2\sum_{n=1}^{m}\left \vert A_{n}\right \vert .
\]
Which further implies that%
\begin{equation}
\sum_{n=1}^{m}\left \vert A_{n}\right \vert +\left(  \frac{2\left(
\gamma+\alpha \right)  -\left \vert \beta \right \vert }{2\left \vert
\beta \right \vert }\right)  \sum_{n=m+1}^{\infty}\left \vert A_{n}\right \vert
\leq1. \label{u}%
\end{equation}
It suffices to show that the left hand side of (\ref{u}) is bounded above by
$\left(  \frac{2\left(  \gamma+\alpha \right)  -\left \vert \beta \right \vert
}{2\left \vert \beta \right \vert }\right)  \sum_{n=1}^{\infty}\left \vert
A_{n}\right \vert ,$which is equivalent to%
\[
\left(  \frac{2\left(  \gamma+\alpha \right)  -3\left \vert \beta \right \vert
}{2\left \vert \beta \right \vert }\right)  \sum_{n=1}^{m}\left \vert
A_{n}\right \vert \geq0.
\]

To prove (\ref{tt}), we write%
\begin{align*}
\frac{1+w(z)}{1-w(z)}  &  =\left(  \frac{2\left(  \gamma+\alpha \right)
+\left \vert \beta \right \vert }{2\left \vert \beta \right \vert }\right)  \left[
\frac{(\mathbb{R}_{\alpha,\beta,\gamma})_{m}(z)}{\mathbb{R}_{\alpha
,\beta,\gamma}(z)}-\frac{2\left(  \gamma+\alpha \right)  -\left \vert
\beta \right \vert }{2\left(  \gamma+\alpha \right)  +\left \vert \beta \right \vert
}\right] \\
&  =\frac{1+\sum_{n=1}^{m}A_{n}z^{n}-\left(  \frac{2\left(  \gamma
+\alpha \right)  +\left \vert \beta \right \vert }{2\left \vert \beta \right \vert
}\right)  \sum_{n=m+1}^{\infty}A_{n}z^{n}}{1+\sum_{n=1}^{\infty}A_{n}z^{n}}.
\end{align*}
Therefore%
\[
\left \vert w(z)\right \vert \leq \frac{\left(  \frac{2\left(  \gamma
+\alpha \right)  +\left \vert \beta \right \vert }{2\left \vert \beta \right \vert
}\right)  \sum_{n=m+1}^{\infty}\left \vert A_{n}\right \vert }{2-2\sum_{n=1}%
^{m}\left \vert A_{n}\right \vert -\left(  \frac{2\left(  \gamma+\alpha \right)
-3\left \vert \beta \right \vert }{2\left \vert \beta \right \vert }\right)
\sum_{n=m+1}^{\infty}\left \vert A_{n}\right \vert }\leq1.
\]
The last inequality is equivalent to%
\begin{equation}
\sum_{n=1}^{m}\left \vert A_{n}\right \vert +\left(  \frac{2\left(
\gamma+\alpha \right)  -\left \vert \beta \right \vert }{2\left \vert
\beta \right \vert }\right)  \sum_{n=m+1}^{\infty}\left \vert A_{n}\right \vert
\leq1. \label{SD}%
\end{equation}
Since the left hand side of (\ref{SD}) is bounded above by $\left(
\frac{2\left(  \gamma+\alpha \right)  -\left \vert \beta \right \vert
}{2\left \vert \beta \right \vert }\right)  \sum_{n=1}^{\infty}\left \vert
A_{n}\right \vert $, this completes the proof.
\end{proof}

\begin{theorem}
\label{th2}Let $\alpha \geq0,\gamma \geq1$ and $\gamma+\alpha \geq3\left \vert
\beta \right \vert $. Then
\end{theorem}%

\begin{equation}
\mathfrak{R}\left \{  \frac{\mathbb{R}_{\alpha,\beta,\gamma}^{\prime}%
(z)}{(\mathbb{R}_{\alpha,\beta,\gamma})_{m}^{\prime}(z)}\right \}  \geq
\frac{\gamma+\alpha-3\left \vert \beta \right \vert }{\gamma+\alpha-\left \vert
\beta \right \vert },\text{ \  \  \ }z\in \mathfrak{U}, \label{kl}%
\end{equation}

and
\begin{equation}
\mathfrak{R}\left \{  \frac{(\mathbb{R}_{\alpha,\beta,\gamma})_{m}^{\prime}%
(z)}{\mathbb{R}_{\alpha,\beta,\gamma}^{\prime}(z)}\right \}  \geq \frac
{\gamma+\alpha-\left \vert \beta \right \vert }{\gamma+\alpha+\left \vert
\beta \right \vert },\text{ \  \  \  \ }z\in \mathfrak{U}. \label{mmy}%
\end{equation}

\begin{proof}
From part (ii) of Lemma \ref{lem2}, we observe that%
\[
1+\sum_{n=1}^{\infty}(n+1)\left \vert A_{n}\right \vert \leq \frac{\gamma
+\alpha+\left \vert \beta \right \vert }{\gamma+\alpha-\left \vert \beta
\right \vert },
\]
where $A_{n}=\frac{\beta^{n}\Gamma \left(  \gamma+\alpha \right)  }%
{\Gamma \left(  \left(  \gamma+\alpha \right)  (n+1)\right)  }$. This implies
that%
\[
\left(  \frac{\gamma+\alpha-\left \vert \beta \right \vert }{2\left \vert
\beta \right \vert }\right)  \sum_{n=1}^{\infty}(n+1)\left \vert A_{n}\right \vert
\leq1.
\]
Consider%
\begin{align*}
&  \frac{1+w(z)}{1-w(z)}\\
&  =\left(  \frac{\gamma+\alpha-\left \vert \beta \right \vert }{2\left \vert
\beta \right \vert }\right)  \left[  \frac{\mathbb{R}_{\alpha,\beta,\gamma
}^{\prime}(z)}{(\mathbb{R}_{\alpha,\beta,\gamma})_{m}^{\prime}(z)}%
-\frac{\gamma+\alpha-3\left \vert \beta \right \vert }{\gamma+\alpha-\left \vert
\beta \right \vert }\right] \\
&  =\frac{1+\sum_{n=1}^{m}(n+1)A_{n}z^{n}+\left(  \frac{\gamma+\alpha
-\left \vert \beta \right \vert }{2\left \vert \beta \right \vert }\right)
\sum_{n=m+1}^{\infty}(n+1)A_{n}z^{n}}{1+\sum_{n=1}^{m}(n+1)A_{n}z^{n}}.
\end{align*}

Therefore%
\[
|w(z)|\leq \frac{\left(  \frac{\gamma+\alpha-\left \vert \beta \right \vert
}{2\left \vert \beta \right \vert }\right)  \sum_{n=m+1}^{\infty}(n+1)\left \vert
A_{n}\right \vert }{2-2\sum_{n=1}^{m}(n+1)\left \vert A_{n}\right \vert -\left(
\frac{\gamma+\alpha-\left \vert \beta \right \vert }{2\left \vert \beta \right \vert
}\right)  \sum_{n=m+1}^{\infty}(n+1)\left \vert A_{n}\right \vert }\leq1.
\]
The last inequality is equivalent to%
\begin{equation}
\sum_{n=1}^{m}(n+1)\left \vert A_{n}\right \vert +\left(  \frac{\gamma
+\alpha-\left \vert \beta \right \vert }{2\left \vert \beta \right \vert }\right)
\sum_{n=m+1}^{\infty}(n+1)\left \vert A_{n}\right \vert \leq1. \label{ii}%
\end{equation}
It suffices to show that the left hand side of (\ref{ii}) is bounded above by%
\[
\left(  \frac{\gamma+\alpha-\left \vert \beta \right \vert }{2\left \vert
\beta \right \vert }\right)  \sum_{n=1}^{\infty}(n+1)\left \vert A_{n}%
\right \vert
\]
which is equivalent to%
\[
\left(  \frac{\gamma+\alpha-3\left \vert \beta \right \vert }{2\left \vert
\beta \right \vert }\right)  \sum_{n=1}^{m}(n+1)\left \vert A_{n}\right \vert
\geq0.
\]

To prove the result (\ref{mmy}), we write%
\begin{align*}
&  \frac{1+w(z)}{1-w(z)}\\
&  =\left(  \frac{\gamma+\alpha+\left \vert \beta \right \vert }{2\left \vert
\beta \right \vert }\right)  \left[  \frac{(\mathbb{R}_{\alpha,\beta,\gamma
})_{m}^{\prime}(z)}{\mathbb{R}_{\alpha,\beta,\gamma}^{\prime}(z)}-\frac
{\gamma+\alpha-\left \vert \beta \right \vert }{\gamma+\alpha+\left \vert
\beta \right \vert }\right]
\end{align*}
where%
\[
\left \vert w(z)\right \vert \leq \frac{\left(  \frac{\gamma+\alpha+\left \vert
\beta \right \vert }{2\left \vert \beta \right \vert }\right)  \sum_{n=m+1}%
^{\infty}(n+1)\left \vert A_{n}\right \vert }{2-2\sum_{n=1}^{m}(n+1)\left \vert
A_{n}\right \vert -\frac{\gamma+\alpha-3\left \vert \beta \right \vert
}{2\left \vert \beta \right \vert }\sum_{n=m+1}^{\infty}(n+1)\left \vert
A_{n}\right \vert }\leq1.
\]
The last inequality is equivalent to%
\begin{equation}
\sum_{n=1}^{m}(n+1)\left \vert A_{n}\right \vert +\left(  \frac{\gamma
+\alpha-\left \vert \beta \right \vert }{2\left \vert \beta \right \vert }\right)
\sum_{n=m+1}^{\infty}(n+1)\left \vert A_{n}\right \vert \leq1. \label{pp}%
\end{equation}
It suffices to show that the left hand side of (\ref{pp}) is bounded above by%
\[
\left(  \frac{\gamma+\alpha-\left \vert \beta \right \vert }{2\left \vert
\beta \right \vert }\right)  \sum_{n=1}^{\infty}(n+1)\left \vert A_{n}\right \vert
.
\]

\end{proof}

\begin{theorem}
\label{th3}Let $\alpha \geq0,\gamma \geq1$ and $\gamma+\alpha \geq \left \vert
\beta \right \vert $. Then
\end{theorem}%

\begin{equation}
\mathfrak{R}\left \{  \frac{\mathbb{I}[\mathbb{R}_{\alpha,\beta,\gamma}%
](z)}{(\mathbb{I}[\mathbb{R}_{\alpha,\beta,\gamma}])_{m}(z)}\right \}
\geq \frac{2(\gamma+\alpha)-2\left \vert \beta \right \vert }{2(\gamma
+\alpha)-\left \vert \beta \right \vert },\text{ \  \  \  \  \  \ }z\in \mathfrak{U},
\label{re}%
\end{equation}

and%

\begin{equation}
\mathfrak{R}\left \{  \frac{(\mathbb{I}[\mathbb{R}_{\alpha,\beta,\gamma}%
])_{m}(z)}{\mathbb{I}[\mathbb{R}_{\alpha,\beta,\gamma}](z)}\right \}  \geq
\frac{2(\gamma+\alpha)-\left \vert \beta \right \vert }{2(\gamma+\alpha)},\text{
\  \  \  \  \  \  \ }z\in \mathfrak{U}. \label{oo}%
\end{equation}
where $\mathbb{I}[\mathbb{R}_{\alpha,\beta,\gamma}]$ is the Alexander
transform of $\mathbb{R}_{\alpha,\beta,\gamma}.$

\begin{proof}
To prove (\ref{re}) , we consider from part (iii) of Lemma \ref{lem2} so that%
\[
1+\sum_{n=1}^{\infty}\frac{\left \vert A_{n}\right \vert }{n+1}\leq
\frac{2(\gamma+\alpha)}{2(\gamma+\alpha)-\left \vert \beta \right \vert },
\]
or equivalently%
\[
\left(  \frac{2(\gamma+\alpha)-\left \vert \beta \right \vert }{\left \vert
\beta \right \vert }\right)  \sum_{n=1}^{\infty}\frac{\left \vert A_{n}%
\right \vert }{n+1}\leq1
\]

where $A_{n}=\frac{\beta^{n}\Gamma \left(  \gamma+\alpha \right)  }%
{\Gamma \left(  \left(  \gamma+\alpha \right)  (n+1)\right)  }.$Now, we write%
\begin{align}
\frac{1+w(z)}{1-w(z)}  &  =\left(  \frac{2(\gamma+\alpha)-\left \vert
\beta \right \vert }{\left \vert \beta \right \vert }\right)  \left[
\frac{\mathbb{I}[\mathbb{R}_{\alpha,\beta,\gamma}](z)}{(\mathbb{I}%
[\mathbb{R}_{\alpha,\beta,\gamma}])_{m}(z)}-\frac{2(\gamma+\alpha)-2\left \vert
\beta \right \vert }{2(\gamma+\alpha)-\left \vert \beta \right \vert }\right]
\nonumber \\
&  =\frac{1+\sum_{n=1}^{m}\frac{\left \vert A_{n}\right \vert }{n+1}%
z^{n}+\left(  \frac{2(\gamma+\alpha)-\left \vert \beta \right \vert }{\left \vert
\beta \right \vert }\right)  \sum_{n=m+1}^{\infty}\frac{\left \vert
A_{n}\right \vert }{n+1}z^{n}}{1+\sum_{n=1}^{m}\frac{\left \vert A_{n}%
\right \vert }{n+1}z^{n}}. \label{mmm}%
\end{align}
Now, from (\ref{mmm}) we can write%
\[
w(z)=\frac{\left(  \frac{2(\gamma+\alpha)-\left \vert \beta \right \vert
}{\left \vert \beta \right \vert }\right)  \sum_{n=m+1}^{\infty}\frac{A_{n}}%
{n+1}z^{n}}{2+2\sum_{n=1}^{m}\frac{A_{n}}{n+1}z^{n}+\left(  \frac
{2(\gamma+\alpha)-\left \vert \beta \right \vert }{\left \vert \beta \right \vert
}\right)  \sum_{n=m+1}^{\infty}\frac{A_{n}}{n+1}z^{n}}.
\]
Using the fact that $|w(z)|\leq1$, we get%
\[
\left \vert \frac{\left(  \frac{2(\gamma+\alpha)-\left \vert \beta \right \vert
}{\left \vert \beta \right \vert }\right)  \sum_{n=m+1}^{\infty}\frac{\left \vert
A_{n}\right \vert }{n+1}}{2-2\sum_{n=1}^{m}\frac{\left \vert A_{n}\right \vert
}{n+1}-\left(  \frac{2(\gamma+\alpha)-\left \vert \beta \right \vert }{\left \vert
\beta \right \vert }\right)  \sum_{n=m+1}^{\infty}\frac{\left \vert
A_{n}\right \vert }{n+1}}\right \vert \leq1.
\]
The last inequality is equivalent to%
\begin{equation}
\sum_{n=1}^{m}\frac{\left \vert A_{n}\right \vert }{n+1}+\left(  \frac
{2(\gamma+\alpha)-\left \vert \beta \right \vert }{\left \vert \beta \right \vert
}\right)  \sum_{n=m+1}^{\infty}\frac{\left \vert A_{n}\right \vert }{n+1}\leq1.
\label{kk}%
\end{equation}
It suffices to show that the left hand side of (\ref{kk}) is bounded above by
$\left(  \frac{2(\gamma+\alpha)-\left \vert \beta \right \vert }{\left \vert
\beta \right \vert }\right)  \sum_{n=1}^{\infty}\frac{\left \vert A_{n}%
\right \vert }{n+1},$which is equivalent to%
\[
\left(  \frac{2(\gamma+\alpha)-2\left \vert \beta \right \vert }{\left \vert
\beta \right \vert }\right)  \sum_{n=1}^{m}\frac{\left \vert A_{n}\right \vert
}{n+1}\geq0.
\]

The proof of (\ref{oo}) is similar to the proof of Theorem \ref{th1}.\bigskip
\end{proof}

\section{ Special cases}

In this section, we obtain the following corollaries for special cases of
$\alpha,\beta$ and $\gamma$ for the functions listed in Remark \ref{rem1}.

If we take $m=0,$ $\gamma=1,$ $\alpha=0$ and $\beta=-\frac{1}{3}$ in Theorem
\ref{th1} and Theorem \ref{th2}, we obtain the following corollary.

\begin{corollary}
The following inequalities hold true:
\end{corollary}

\[
\mathfrak{R}\left \{  e^{-\frac{z}{3}}\right \}  \geq \frac{3}{5}=\allowbreak
0.6,\text{ \  \  \  \  \  \ }z\in \mathfrak{U},
\]

\[
\mathfrak{R}\left \{  e^{\frac{z}{3}}\right \}  \geq \frac{5}{7}\approx
0.714\,29,\text{ \  \  \  \  \  \  \ }z\in \mathfrak{U},
\]

\[
\mathfrak{R}\left \{  -\frac{1}{3}e^{-\frac{z}{3}}(z-3)\right \}  \geq0,\text{
\  \  \ }z\in \mathfrak{U},
\]

and
\[
\mathfrak{R}\left \{  -\frac{3e^{\frac{z}{3}}}{z-3}\right \}  \geq0.5,\text{
\  \  \  \ }z\in \mathfrak{U}.
\]
If we take $m=0,$ $\gamma=1,$ $\alpha=1$ and $\beta=\frac{1}{2}$ in Theorem
\ref{th1} and Theorem \ref{th2}, we obtain the following corollary.

\begin{corollary}
The following inequalities hold true:
\end{corollary}

\[
\mathfrak{R}\left \{  \sqrt{\frac{2}{z}}\sinh \sqrt{\frac{z}{2}}\right \}
\geq \frac{5}{7}\allowbreak \approx0.714\,29,\text{ \  \  \  \  \  \ }z\in
\mathfrak{U},
\]%
\[
\mathfrak{R}\left \{  \sqrt{\frac{z}{2}}\csc h\sqrt{\frac{z}{2}}\right \}
\geq \frac{7}{9}\approx0.777\,78,\text{ \  \  \  \  \  \  \ }z\in \mathfrak{U},
\]

\[
\mathfrak{R}\left \{  \frac{1}{2}\cosh \sqrt{\frac{z}{2}}+\frac{\sinh \sqrt
{\frac{z}{2}}}{\sqrt{2z}}\right \}  \geq \frac{1}{3}\approx0.333\,33,\text{
\  \  \ }z\in \mathfrak{U},
\]

and
\[
\mathfrak{R}\left \{  \frac{1}{\frac{1}{2}\cosh \sqrt{\frac{z}{2}}+\frac
{\sinh \sqrt{\frac{z}{2}}}{\sqrt{2z}}}\right \}  \geq \frac{3}{5}=0.6,\text{
\  \  \  \ }z\in \mathfrak{U}.
\]
If we take $m=0,$ $\gamma=1,$ $\alpha=1$ and $\beta=-\frac{1}{4}$ in Theorem
\ref{th1} and Theorem \ref{th2}, we obtain the following corollary.

\begin{corollary}
The following inequalities hold true:
\end{corollary}

\[
\mathfrak{R}\left \{  \frac{2}{\sqrt{z}}\sin \frac{\sqrt{z}}{2}\right \}
\geq \frac{13}{15}\allowbreak \approx0.866\,67,\text{ \  \  \  \  \  \ }%
z\in \mathfrak{U},
\]

\[
\mathfrak{R}\left \{  \frac{\sqrt{z}}{2}\csc \frac{\sqrt{z}}{2}\right \}
\geq \frac{15}{17}\approx0.882\,35,\text{ \  \  \  \  \  \  \ }z\in \mathfrak{U},
\]

\[
\mathfrak{R}\left \{  \frac{1}{2}\cos \frac{\sqrt{z}}{2}+\frac{\sin \frac
{\sqrt{z}}{2}}{\sqrt{z}}\right \}  \geq \frac{5}{7}\approx0.714\,29,\text{
\  \  \ }z\in \mathfrak{U},
\]

and
\[
\mathfrak{R}\left \{  \frac{1}{\frac{1}{2}\cos \frac{\sqrt{z}}{2}+\frac
{\sin \frac{\sqrt{z}}{2}}{\sqrt{z}}}\right \}  \geq \frac{7}{9}\approx
0.777\,78,\text{ \  \  \  \ }z\in \mathfrak{U}.
\]
If we take $m=0,$ $\gamma=1,$ $\alpha=1$ and $\beta=1$ in Theorem \ref{th1},
we obtain the following corollary.

\begin{corollary}
The following inequalities hold true:%
\[
\mathfrak{R}\left \{  \frac{\sinh \sqrt{z}}{\sqrt{z}}\right \}  \geq \frac{1}%
{3}\allowbreak \approx0.333\,33,\text{ \  \  \  \  \  \ }z\in \mathfrak{U},
\]

\end{corollary}

and%

\[
\mathfrak{R}\left \{  \sqrt{z}\sec h\sqrt{z}\right \}  \geq \frac{3}%
{5}=0.6,\text{ \  \  \  \  \  \  \ }z\in \mathfrak{U}.
\]

\begin{remark}
Putting $m=0$ in inequality (\ref{kl}), we obtain $\mathfrak{R}\left \{
\mathbb{R}_{\alpha,\beta,\gamma}^{\prime}(z)\right \}  >0$. In view of
Noshiro-Warschawski Theorem (see \cite{go}), we have that the normalized
Rabotnov function is univalent in $\mathfrak{U}$ for $\gamma+\alpha
\geq3\left \vert \beta \right \vert .$
\end{remark}

\section{Conclusions}

Recently, some researchers have studied on partial sums of special functions
like, the normalized Struve functions, the normalized Dini functions, the
normalized Wright functions and the normalized Miller-Ross function. In this
present paper, we have introduced a new generalization of the Rabotnov
function $R_{\alpha,\beta,\gamma}$ given by (\ref{gg}). Furthermore, we
determined lower bounds for the normalized Rabotnov function $\mathbb{R}%
_{\alpha,\beta,\gamma}$ to its sequence of partial sums. Several interesting
examples of the main results are also considered.

\end{document}